\documentclass{amsart}
\usepackage{amsaddr}
\usepackage[utf8]{inputenc}
\usepackage{latexsym,amssymb}
\usepackage{amsmath}
\usepackage{tabularx}
\usepackage[matrix,arrow,curve]{xy}
\usepackage{hyperref}

\numberwithin{equation}{section}
\newtheorem{theorem}{Theorem}[section]
\newtheorem{lemma}[theorem]{Lemma}
\newtheorem{proposition}[theorem]{Proposition}

\theoremstyle{definition}
\newtheorem{definition}[theorem]{Definition}

\newtheorem{remark}[theorem]{Remark}
\newtheorem{example}[theorem]{Example}

\newcommand{\NN}{\mathbb{N}}
\newcommand{\ZZ}{\mathbb{Z}}

\newcommand{\CC}{\mathbb{C}}

\renewcommand{\AA}{\mathbb{A}}

\newcommand{\PP}{\mathbb{P}}
\newcommand{\FF}{\mathcal{F}}

\newcommand{\Q}{\mathcal{Q}}

\newcommand{\OO}{\mathcal{O}}
\newcommand{\II}{\mathcal{I}}
\newcommand{\HHom}{\mathcal{H}om}

\newcommand{\Sym}{\operatorname{Sym}}
\newcommand{\IN}{\operatorname{in}}
\newcommand{\codim}{\operatorname{codim}}

\newcommand{\HH}{\mathbb{H}}
\newcommand{\BH}{\operatorname{BH}}

\renewcommand{\OO}{\mathcal{O}}
\newcommand{\EE}{\mathcal{E}}
\newcommand{\UU}{\mathcal{U}}
\newcommand{\RR}{\mathcal{R}}
\renewcommand{\FF}{\mathcal{F}}

\newcommand{\Spec}{\operatorname{Spec}}

\newcommand{\Proj}{\operatorname{Proj}}

\newcommand{\Hom}{\operatorname{Hom}}
\newcommand{\Ext}{\operatorname{Ext}}
\newcommand{\CHom}{\operatorname{CHom}}
\newcommand{\GL}{\operatorname{GL}}
\newcommand{\Gr}{\operatorname{Gr}}
\newcommand{\oGr}{\operatorname{OGr}}
\newcommand{\LL}{\mathcal{L}}
\newcommand{\MM}{\mathcal{M}}
\newcommand{\tr}{\frac{1}{d}\operatorname{tr}}

\begin{document}

\title[Construction of algebraic covers]{Construction of algebraic covers \\ \tiny{Cover Homomorphisms}}

\author{Eduardo Dias}
\address{CMUP, Centro de Matemática da Universidade do Porto,  Portugal}
\email{eduardo.dias@fc.up.pt}

\thanks{This research was supported by FCT (Portugal) under the project 
PTDC/MAT-GEO/2823/2014 and by CMUP (UID/MAT/00144/2019),
which is funded by FCT with national (MCTES) and European structural funds through the programs FEDER,
under the partnership agreement PT2020.}

\begin{abstract}
Let $Y$ be an algebraic variety, $\FF$ a locally free sheaf of $\OO_Y$-modules, and $\RR(\FF)$ the $\OO_Y$-algebra $\Sym^\bullet \FF$. In this paper we study local properties of sheaves of $\OO_{\RR(\FF)}$-ideals $\II$ such that $\RR(\FF)/\II$ is an algebraic cover of $Y$. 
Following the work of Miranda for triple covers, for $\Q$ a direct summand of $\RR(\FF)$, we say that a morphism $\Phi\colon \Q\rightarrow\RR(\FF)/\langle\Q\rangle$ is a covering homomorphism if it induces such an ideal. 
As an application we study in detail the case of Gorenstein covering maps of degree $6$ for which the direct image of $\varphi_*\mathcal{O}_X$ admits an orthogonal decomposition. These are deformation of $S_3$-Galois branch covers.
\end{abstract}

\maketitle

\section{Introduction}\label{Introduction}
The aim of this article is to investigate the algebra of covering maps in algebraic geometry, i.e. finite flat morphisms of degree $d$ between algebraic varieties. Given a covering map $\varphi\colon X\rightarrow Y$, $\varphi_*\OO_X$ is a coherent sheaf of $\OO_Y$-algebras whose $\OO_Y$-structure corresponds to an associative and commutative map 
$m\colon \varphi_*\OO_X\otimes_{\OO_Y}\varphi_*\OO_X\rightarrow\varphi_*\OO_X$. Our goal is to study this multiplication and Theorem \ref{thmglobal} is essentially the following.

\begin{theorem}
A commutative and associative $\OO_Y$-algebra structure on $\varphi_*\OO_X$, for $\varphi\colon X\rightarrow Y$ a covering map, is given by a pair $(\FF,\Phi)$, where 
\begin{enumerate}
    \item $\FF$ is a locally free sheaf of $\OO_Y$-modules such that $\varphi_*\OO_X$ is a direct summand of $\RR(\FF):=\bigoplus_{n=0}^\infty \Sym^n\FF$;
    \item $\Phi\in \Hom(\Q,\varphi_*\OO_X)$ is a \textit{covering homomorphism}, where $\Q$ is a direct summand of $\RR(\FF)$ such that $\varphi_*\OO_X=\RR(\FF)/\langle\Q\rangle$.
\end{enumerate}
\end{theorem}

The morphism $\Phi\in \Hom(\Q,\varphi_*\OO_X)$ is a covering homomorphism if, for any local basis $\{q_j\}$ of $\Q$, the sheaf of $\OO_{\RR(\FF)}$-ideals locally generated by the polynomials
\begin{equation}\label{eqintro}
    q_j-\Phi(q_i)
\end{equation} 
is a sheaf of $\OO_{\RR(\FF)}$-ideals with the same codimension as $\langle Q\rangle$. If we denote by $\II_X$ this ideal, then $\OO_X$ is the scheme $\RR(\FF)/\II_X$. 
The multiplication of elements in this scheme is trivially commutative and associative.

This result follows the approach of Miranda for triple covers in \cite{triple}. After proving that for a covering map of degree $3$ we have a decomposition $\varphi_*\OO_X=\OO_Y\oplus\EE$, where $\EE$ is a locally free sheaf of $\OO_Y$-modules with rank $2$, Miranda goes further and proves that, amongst other important results, an associative multiplication in $\varphi_*\OO_X$ is determined by a morphism $\Phi\in\Hom(\Sym^2\EE,\EE)$. The morphism from $\Sym^2\EE$ to $\OO_Y$ is constructed from $\Phi$.

In the same line of thought, in Proposition \ref{quadeq}, we prove that if $\Q$ is a direct summand of $\Sym^2\FF$ and the first syzygy matrix of the ideal defined by the equations (\ref{eqintro}) is linear, then a covering homomorphism is determined by a morphism in $\Hom(\Q,\FF)$.

Looking at the local resolution of a of $\RR(\FF)$-ideals defining $\OO_X$ is the difference of our approach from the method used for triple covers and quadruple covers by Miranda and Hahn, \cite{quadruple}. Important to mention that it is not new, Miranda and Pardini in \cite{PardiniTriple} notice that triple covers are locally determinantal varieties. We will exploit this fact in Section \ref{algebraiccoverssection} and, instead of looking at cover homomorphisms as morphisms that determine an associative multiplication, define a cover homomorphism as one that determines a sheaf of $\RR(\FF)$-ideals that are a flat deformation of an ideal sheaf $\langle\Q\rangle\subset \RR(\FF)$. This approach uses the theory of \textit{"Hilbert scheme of extensions"} of a given ring, following the article \textit{Infinitesimal view if extending a hyperplane section -- deformation theory and computer algebra} by Miles Reid.

In Section \ref{GorensteinCovers} we apply the method to Gorenstein covering maps. An algebraic cover $\varphi\colon X\rightarrow Y$ is called Gorenstein if the fibre $X_y$ is Gorenstein for every point $y\in Y$. Our method is based on the global structure of $\varphi_*\OO_X$ so, in Theorem \ref{socle}, we prove that if there is a line bumble $\LL$ on $Y$ such that $\omega_Y=\LL^{k_Y}$ and $\omega_X=\varphi^*\LL^{k_X}$, for integers $k_Y, k_X$, then $\varphi_*\OO_X$ is isomorphic to its dual. In particular $\varphi_*\OO_X$ decomposes as 
$$\varphi_*\OO_X= \OO_Y\oplus \FF\oplus\MM$$ 
for a sheaf $\FF$ and a line bundle $\MM$. There exists a structure theorem for these algebraic covers by Casnati and Ekedahl in \cite{casnatiI}. Our result does not apply to every Gorenstein covering map as their does, but, if $X$ and $Y$ satisfy the property described, the Theorem can be used to explicitly describe the equations defining $X$.

In Section \ref{Codimension4} we analyse a Gorenstein covering map of degree $6$ satisfying 
$$\varphi_*\OO_X=\OO_Y\oplus M\oplus M\oplus \wedge^2M,$$ 
where $M$ is a simple $\OO_Y$-sheaf of rank $2$. Although its fibre is given by a codimension $4$ Gorenstein ideal, codimension for which there is no structure theorem, the conditions imposed on $\varphi_*\OO_X$ enable us to determine the local equations defining the multiplication in $\varphi_*\OO_X$ via an explicit SageMath \cite{SAGE} algorithm (presented in the Appendix). Theorem \ref{maintheorem4} states that this local structure is given by the polynomials whose vanishing defines the spinor embedding of the orthogonal Grassmann variety $\oGr(5,10)$. 

%The motivation for this example is the study of the coordinate ring of abelian surfaces with a polarisation of type $(1,3)$. These are degree $6$ Gorenstein covers of $\PP^2$ as above, with $M\cong \Omega^1_{\PP^2}$.

In Theorem \ref{galoisthm} we study a linear component of these degree $6$ Gorenstein covers and see that they are a family of $S_3$-Galois branched covers. The algebraic structure of Galois $G$-covers, for an abelian group $G$, has been famously studied by Pardini in \cite{PardiniAbelian}. In that case one can explicitly describe the algebraic cover by the branching data satisfying covering conditions. Non-abelian $G$-covers have been studied by Tokunaga for $G$ a Dihedral group \cite{DihedralCovers,DihedralCoversApp} and the case of $S_3$-covers was examined before by Easton in \cite{S3covers}.

Hahn and Miranda found that the local structure for covering spaces of degree $4$ is given by the polynomials defining the Pl\"uker embedding of the Grassmanian $\Gr(2,6)$ in $\PP^{14}$. This result and the connection of Gorenstein covers of degree $6$ and $\oGr(5,10)$ might happen by chance, i.e. smooth structure of the system of parameters and the right codimension may imply the connection with these Grassmann varieties. We believe that there is a geometric relation and understanding it might be the perspective needed to study higher codimension Gorenstein varieties.

Throughout the paper $k$ will denote a characteristic $0$ algebraically closed field.

\section{Algebraic Covers}\label{algebraiccoverssection}
\begin{definition}
Let $X, Y$ be schemes over $k$.
A covering map of degree $d$ is a flat and finite morphism $\varphi\colon X \rightarrow Y$, such that $\varphi_*\OO_X$ is a sheaf of locally free $\OO_Y$-modules with rank $d$.
\end{definition}

A well known result is that for a covering map of degree $d$, the trace map $\tr\colon\varphi_*\OO_X\rightarrow\OO_Y$ gives rise to a splitting $\varphi_*\OO_X=\OO_Y\oplus \EE$, where $\EE=\ker\left(\tr\right)$.
The result is true for fields with characteristic prime with $d$ but, as we want our result to be applied for general exponents, we will stick to the characteristic $0$ case.

\subsection{Local analysis}\label{localanalysis}

Assume in this section that $Y$ is a local scheme over $k$, \mbox{specifically}, let $\OO_Y$ be a local, Noetherian, $k$--algebra and $\OO_X$ a free sheaf of $\OO_Y$--algebras with rank $d$. Furthermore, let $\FF$ be a free sheaf of $\OO_Y$--modules such that $\OO_X$ is a \textit{direct summand} of $\RR(\FF)\cong \bigoplus_{n=0}^\infty \Sym^n\FF$, i.e.
there is a surjective morphism 
$$
\pi\colon \bigoplus_{n=0}^\infty \Sym^n\FF= \RR(\FF)\rightarrow \OO_X,
$$
and a section $$\sigma\colon \OO_X\rightarrow \RR(\FF).$$ 
We have then a isomorphism $\OO_X\cong \RR(\FF)/ \II_X$, where $\II_X$ is a sheaf of $\RR(\FF)$-ideals. Our method to describe the algebra structure on $\OO_X$, given its description as an $\OO_Y$--module, is to study, for a given $\FF$, all possible sheaves of ideals $\II_X$.

Let $\{z_1,\dots,z_r\}$ be a choice of basis for $\FF$ so that $\RR(\FF)\cong\OO_Y[z_i]$. Given an element $f\in\RR(\FF)$ we use this isomorphism to define $\operatorname{in}(f)$ as the sum of the terms with maximal degree in the variables $z_i$ of $f$, and the initial ideal of $\II_X$ as
$$\operatorname{in}(\II_X)=\langle \operatorname{in}(f)\colon f\in \II_X\rangle.$$

In the next Lemma we prove that the ideal $\operatorname{in}(\II_X)$ is well defined, i.e. for different basis chosen for $\FF$, the ideals $\IN(\II_X)$ are isomorphic. 

\begin{lemma}\label{CMvsAss}
Let $\OO_Y$ be a Noetherian local $k$-algebra, $\OO_X$ a finite flat $\OO_Y$--algebra, and $\FF$ a free $\OO_Y$--module with rank $r$ such that $\OO_X$ is a direct summand of $\RR(\FF)$. Then 
\begin{enumerate}
 \item\label{1} the ideal sheaf $\II_X$ is Cohen-Macaulay;
 \item\label{1,5} the ideal $\IN(\II_X)$ is well defined;
 \item\label{2} the free resolution of the ideal $\operatorname{in}(\II_X)$ has the same length as the one for $\II_X$, in particular $\operatorname{in}(\II_X)$ is also CM.
\end{enumerate}
\end{lemma}

\begin{proof}
\ref{1}:
As $\OO_X$ is a finite flat $\OO_Y$-algebra, $\dim \OO_X=\dim \OO_Y\Rightarrow\codim(\II_X)=r$. Take $n\in\NN$ large enough such that $\Sym^n\FF \cap\sigma\left(\OO_X\right)=\{0\}$ (consider the isomorphism $\RR(\FF)\cong \OO_Y[z_i]$ and take $n\gg 0$ such that $\sigma\left(\OO_X\right)$ is generated by $z_1^{i_1}z_2^{i_2}\cdots z_r^{i_r}$ with $\sum i_j< n$). 

Then the sequence $(r_j):=(z_j^n-\sigma\circ \pi(z_j^n))_{1\leq j\leq r}$ is an $\RR\left(\FF\right)$--sequence, just check that the quotient $\OO_Y[z_i]/\langle z_j\rangle_{j=1}^k$ is a finitely generated algebra over the ring $\OO_Y[z_i]_{i=k+1}^r$, contained in the ideal $\operatorname{in}(\II_X)$. we conclude that
$$r\leq \operatorname{depth}(\II_X)\leq \codim(\II_X)=r,$$
which asserts that $\II_X$ is CM.

\ref{1,5}:
Denote by $\IN(\II_X)_{(z_i)}$ the initial ideal when defined with respect to the basis $\{z_i\}$ for $\FF$. A change of basis for $\RR\left(\FF\right)$ is given by a $\OO_Y$-linear map, $\Psi\colon \FF\rightarrow \OO_Y\oplus \FF$. Decomposing $\Psi$ as $\Psi_1\oplus\Psi_2$, where $\Psi_1\colon\FF\rightarrow \OO_Y$ and $\Psi_2\colon\FF\rightarrow\FF$, we have that $\Psi_2$ is an isomorphism which sends $\IN(\II_X)_{\{z_i\}}$ into $\IN(\II_X)_{\{\Psi(z_i)\}}$ so we are done. 

\ref{2}: 
Denote by $\widetilde{\II}_X$ the ideal obtained by the change of variables $z_i\mapsto z_i/z_0$ to homogenise $\II_X$.  As $z_0$ is a non-zero divisor in $\widetilde{\II}_X$, taking the tensor of a minimal free resolution of it with $\OO_Y[z_i]/(z_0)$ gives a resolution of the ideal $\operatorname{in}(\II_X)$, so we have the right $\operatorname{depth}$. As codimension of $\operatorname{in}(\II_X)$ is smaller or equal to $r$ we are done.  
\end{proof}

Consider that we have the injective and surjective morphisms $\OO_Y\hookrightarrow\RR(\FF)\twoheadrightarrow  \OO_X$,
for some free $\OO_Y$--sheaf $\FF$ with basis $\{z_1,\dots, z_r\}$, and suppose that $\{q_j(z_i)\}$ are homogeneous polynomials generating the ideal $\IN(\II_X)$. We want to do the reserve of the trick used to prove part $(\ref{2})$ of the previous Lemma. To be more specific, we want to determine the relations between (token) coefficients $c_{ij}\in\OO_Y$, $(i,j)\in \NN\times\NN^r$, such that the 
ring $R=\OO_Y[z_i]_{i=0}^r/\langle f_j\rangle_j$, where  
$$f_i:= q_i-\sum_j c_{ij}\overline{z}^jz_0^{(\deg(q_i)-\sum_k j_k)},\,\hspace{5mm} \overline{z}^j=z_1^{j_1}\cdots z_r^{j_r},$$
is a flat deformation of $\overline{R}=\OO_Y[z_i]_{i=1}^r/\IN(\II_X)$, i.e. $z_0$ is a non-zero divisor in $R$, of degree $1$, and $\overline{R}= R/\langle z_0\rangle$.

The goal of the next Section is to show a computational algorithm to determine the relations between the $c_{ij}$ given an ideal $\II_X$. Notice that by Nakayama's Lemma, the image of the local generators of $\II_X$ in $\II_X\otimes k$ are the generators of $\II_X\otimes k$. Hence, in the next Section, we use this fact and study the algebra of $d$ points over $\Spec(k)$ defined by polynomials whose terms with maximal degree are known. 

\subsection{Hilbert Scheme of Extensions}\label{iqrelations}

In this section we follow the paper \cite{extending} by Miles Reid, where the following extension problem is tackled.
\textit{Given a Noetherian graded ring $\overline{R}$ and $a_0$ a non negative integer, determine the set of all pairs $(R,z_0)$, where $R$ is a graded ring and $z_0\in R$ is a homogeneous, nonzero divisor, of degree $a_0$ such that $R/(z_0)=\overline{R}$.} We call $R$ an extension of $\overline{R}$.

The hyperplane section principle (see \cite[$\S 1.2$]{extending}) says that the generators of $R$, their relations and syzygies reduce modulo $z_0$ to those of $\overline{R}$ and occur in the same degrees. In particular, if $\overline{R}=\overline{S}/\overline{I}$, where $\overline{S}=k[z_1,\dots,z_r]$, and we have an exact sequence 
\begin{equation}\label{resolutionquot}
0\leftarrow
\overline R\leftarrow \overline{S}\xleftarrow{f}\bigoplus\overline{S}(-a_i) \xleftarrow{l} \bigoplus\overline{S}(-b_j),
\end{equation}
with $a_i,b_j$ positive integers, $f$ a vector with the generators of $\overline{I}$ as entries, $l$ its syzygy matrix, then the resolution of a ring $R$, such that $\overline{R}=R/(z_0)$, starts with the terms 
\begin{equation}\label{resolutionquot2}
0\leftarrow
R\leftarrow S\xleftarrow{F}\bigoplus S(-a_i) \xleftarrow{L} \bigoplus S(-b_j),
\end{equation}
where $S=k[z_0,\dots,z_r]$, $f(z_1,\dots,z_r)=F(0,z_1,\dots,z_r)$, and $l(z_1,\dots,z_r)=L(0,z_1,\dots,z_r)$. The result can be summarised in the following way. 

\begin{proposition}\label{liftresolution}
Given a ring $\overline R=\overline{S}/\overline{I}$ with a presentation as in (\ref{resolutionquot}), if every relation 
$\sigma_j:=\left(\sum_i l_if_i\equiv 0\right)$ 
lifts to a relation $\Sigma_j:=\left(\sum_iL_iF_i\equiv 0\right)$, then the resolution lifts to a resolution of $R$ and $\overline R=R/(z_0)$. 
\end{proposition}

\begin{proof}\cite[$\S 1.2$ - The hyperplane section principle]{extending}
\end{proof}

To compute $R$ from $\overline{R}$ we use its grading. Define a $k^{th}$ order \textit{infinitesimal extension} of $\overline R$
as a ring $\overline R^{(k)}$ together with a homogeneous element $z_0 \in \overline R^{(k)}$ of degree $a_0$ such that $\overline R = \overline R^{(k)}/(z_0)$, $z_0^{k+1} = 0$, and $\overline R^{(k)}$ is flat over the subring $k[z_0]/(z_0^{k+1})$ generated by $z_0$. 

\begin{definition}The Hilbert scheme of $k^{th}$ order infinitesimal extensions of $\overline{R}$ by a variable of degree $a_0$ is the set
 $$\HH^{(k)}\left(\overline{R},a_0\right)=\left\{(R,z_0)\ |\ \deg(z_0)=a_0,\ R/\left(z_0^{k+1}\right)\cong \overline{R}\right\}.$$
\end{definition}

Notice that if $k\geq \max_i\{\deg(\sigma_i)\}$, then $\HH^{(k)}\left(\overline{R},a_0\right)$ is the solution to the extension problem, denoted by $\HH\left(\overline{R},a_0\right)$. It is constructed as a tower of schemes 
$$\HH\left(\overline{R},a_0\right)\cdots\rightarrow \HH^{(k)}\left(\overline{R},a_0\right)\rightarrow \HH^{(k-1)}\left(\overline{R},a_0\right)\rightarrow \cdots \rightarrow \HH^{(0)}\left(\overline{R},a_0\right)=\text{pt},$$
where each morphism is induced by the forgetful map $k[z_0]/\left(z_0^{k+1}\right)\rightarrow k[z_0]/\left(z_0^k\right)$.

The scheme $\HH^{(k)}\left(\overline{R},a_0\right)$ is independent of a choice of variables so we introduce the following scheme for the sake of computations.

\begin{definition}
We call the following set the big Hilbert scheme of $k^{th}$ order infinitesimal \mbox{extensions} of $\overline R$ by a variable $z_0$ of degree $a_0$
$$\BH^{(k)}\left(\overline R,a_0\right)=\left\{\left.\begin{matrix} F_i=f_i+z_0f'_i+ \cdots +z_0^kf_i^{(k)}, \\ L_{ij}=l_{ij}+z_0l'_{ij}+\cdots +z_0^kl^{(k)}_{ij}\end{matrix}\ \right|\ \Sigma_j:\sum_i L_{ij}F_i \equiv 0 \mod z_0^{k+1}\right\}$$
where $F_i, L_{ij}$ are homogeneous polynomials in the ring $\overline S=k[z_0,\dots,z_r]$, $F_i$ generators of the ideal $\overline I$ where $\overline R =\overline S/ \overline I$, and $L_{ij}$ generators of the relations between the $F_i$. If $k\geq \max_j\{\deg(\Sigma_j)\}$, then the set is the Big Hilbert scheme of infinitesimal extensions of $\overline R$ by a \mbox{variable} $z_0$ of degree $a_0$ and it is denoted by $\BH\left(\overline R,a_0\right)$.
The schemes $\BH\left(\overline R,a_0\right)$, and $\BH^{(k)}\left(\overline{R},a_0\right)$ for each $k$, have the structure of an affine scheme with coordinates given by the coefficients of $f^{(m)}_i$ and $l^{(m)}_{ij}$.
\end{definition}

Given a set of generators for the ideal defining a graded ring $\overline R=\overline S/\overline I$ and a degree $a_0>0$, the coordinates of each point on the big Hilbert scheme $\BH\left(\overline R,a_0\right)$ correspond to a set of generators of an ideal $I$, and the relations between these generators, such that $I\otimes S/(z_0)\cong \overline I$.
In other words, the big Hilbert scheme is a choice of coordinates for the Hilbert scheme and hence, the Hilbert scheme is the quotient of the big Hilbert scheme by linear changes of variables. 

The first order extensions of $\overline{R}$,  $\HH^{(1)}(\overline{R},a_0)$, have a natural one-to-one correspondence with the elements in $\Hom(\overline{I}/\overline{I}^2,\overline{R})_{-a_0}$, see \cite[Thm. $1.10$]{extending} or \cite[Thm. $2.4$]{HartshorneDefTh}. To see the relation, notice that $\BH^{(1)}(\overline{R},a_0)$ is given as the set of elements 
$
\{(f_i+z_0f_i'), (l_{ij}+z_0l_{ij}')\}
$
that for all $j$ satisfy 
$$\sum_i l_{ij}'f_i+\sum_i l_{ij}'f_i = 0 \in S(b_j)_{-a_0},$$
recall that $b_j$ is the degree of $\Sigma_j$, see the Exact Sequence $(\ref{resolutionquot})$. As the second summand is an element of $\overline{I}$, we can rewrite the equality as 
$$\sum_i l_{ij}f_i' \in \overline{R}(b_j)_{-a_0}.$$
Hence we can define a $S$-linear map as $f_i\mapsto f_i'$, i.e. an element of 
$$\Hom\left(\overline{I},\overline{R}\right)_{-a_0}=\Hom\left(\overline{I}/\overline{I}^2,\overline{R}\right)_{-a_0}.$$
With the same reasoning, we can identify the fibers of the morphism $\HH^{(k)}(\overline{R},a_0)\rightarrow\HH^{(k-1)}(\overline{R},a_0)$ with $\Hom(\overline{I}/\overline{I}^2,\overline{R})_{-ka_0}$. Now there is not a one-to-one correspondence, we have to check for obstructions. 

Applying the functor $\Hom(-,\overline{R})$ to the free resolution of $\overline{I}$ (Sequence (\ref{resolutionquot}) truncated on the term $\bigoplus S(-a_i)$), we get the complex
$$0\rightarrow \bigoplus \overline{R}(a_i)\xrightarrow{\delta_0} \bigoplus \overline{R}(b_j)\xrightarrow{\delta_1} \cdots $$
whose homology is 
$\Ext^i_{\overline{R}}(\overline{I}/\overline{I}^2,\overline{R})$. In particular there is an exact sequence 
$$0\rightarrow \Hom(\overline{I}/\overline{I}^2,\overline{R})\rightarrow \bigoplus \overline{R}(a_i)\rightarrow \ker \delta_1\rightarrow \Ext^1_{\overline{R}}(\overline{I}/\overline{I}^2,\overline{R}) \rightarrow 0.  $$

This sequence provides structure to the morphism $\varphi_k\colon \HH^{(k)}\left(\overline{R},a_0\right)\rightarrow \HH^{(k-1)}\left(\overline{R},a_0\right)$. To see so, suppose we have $(k-1)$-order polynomials $F_i=\sum_{m=0}^{k-1}z_0^{m}f_i^{(m)}$ and $L_{ij}=\sum_{n=0}^{k-1}z_0^nl_{ij}^{(m)}$ satisfying $\sum L_{ij}F_i\equiv 0 \mod z_0^k$. The extension of these polynomials to order $k$ has new terms $f_i^{(k)}$ and $l_{ij}^{(k)}$ such that the following equality holds
$$ \sum l_{ij}f_i^{(k)}+\sum_{a=1}^{k-1}\sum l_{ij}^{(a)}f_i^{(k-a)}+ \sum l_{ij}^{(k)}f_i=0.$$
As the last term is in $\overline{I}$, it can be rewritten as 
$$ \sum l_{ij}f_i^{(k)}=-\sum_{a=1}^{k-1}\sum l_{ij}^{(a)}f_i^{(k-a)}=:\psi_j\in \overline{R}(b_j)_{-ka_0}.$$

Take $\psi_j$ as components of a morphism $\psi\colon\BH^{(k-1)}(\overline{R},a_0)\rightarrow \overline{R}(b_j)_{-ka_0}$.
In \cite[Thm. $1.15$]{extending}, it is proven that the morphism factors through a morphism $\Psi\colon\HH^{(k-1)}(\overline{R},a_0)\rightarrow \overline{R}(b_j)_{-ka_0}$ and that its image is contained in the kernel of $\delta_1$. 

Putting the results together we get that the middle square in the diagram
\begin{equation}\label{equationMilesmaintheorem}
\begin{array}{l}
\xymatrix{ 
 & \HH^{(k)}\left(\overline{R},a_0\right) \ar[d]\ar[r]^{\varphi_k} & \HH^{(k-1)}\left(\overline{R},a_0\right) \ar[d] \ar[rd]^{\Psi} &  \\
0 \ar[r] & \Hom_{\overline{R}}(\overline{I}/\overline{I}^2,\overline{R})_{-ka_0} \ar[r] & \bigoplus\overline{R}(b_j)_{-ka_0} \ar[r]^{\delta_0} & (\ker \delta_1)_{-ka_0} 
} \\
\hspace{8cm}\xrightarrow{\pi}  \Ext^1_{\overline{R}}(\overline{I}/\overline{I}^2,\overline{R})_{-ka_0} \rightarrow 0
\end{array}
\end{equation}
is Cartesian. In particular, if $\Ext^1_{\overline{R}}(\overline{I}/\overline{I}^2,\overline{R})_{-ka_0}=0$ then there is no obstruction to the extension of $\HH^{(k-1)}(\overline{R},a_0)$ to $\HH^{(k)}(\overline{R},a_0)$.

Back to the context of the paper, we introduce the following definition. 

\begin{definition}\label{iqrelationsdefinition}
 Let $\{q_i\}_{i=1}^m$, be a basis for a homogeneous ideal $q\subset \overline{S}=k[z_j]_{j=1}^r$ such that the ring $\overline{R}=\overline{S}/q$ is a zero-dimensional $k$-algebra. We denote by
 $$\II_q\subset k[c_{ij}],$$ 
 where ${(i,j)\in \NN\times \NN^r}$ and $\sum_ij_i<\deg(q_i)$, the ideal generated by the relations between the $c_{ij}$ which make the ring $R=\overline{S}[z_0]/\langle f_i\rangle$ a flat deformation of $\overline{R}$, where 
 $$f_i:=q_i-\sum c_{ij}\overline{z}^jz_0^{(\deg q_i - \sum j_k)}.$$ 
 We call $\II_q$ the \textit{ideal of $q$-relations}. Notice that $k[c_{ij}]/\II_q\cong \BH(\overline{R},1)$. 
\end{definition}

We now use these tools to describe $\II_q$.
Let $q=(q_1,\dots,q_m)$ be a minimal set of homogeneous generators of an ideal with codimension $r\geq 2$ in $S=k[z_1,\dots,z_r]$. Furthermore, consider that for all $i$, $\deg(q_i)=2$, and the first syzygy matrix of the ideal $(q_1,\dots,q_m)$ is linear. Our goal is to determine the relations between the $c_{ij}$ and $d_{i}$ such that the ring $R=S[z_0]/I$, where
$$I=\left(f_i:=q_i- \left(\sum_{j=1}^rc_{ij}z_j\right)z_0-d_iz_0^2\right)$$
is a flat deformation of $S/q$. As $q_i$ is of degree $2$ and the first syzygy matrix of $q$ is linear, we have the following exact sequence
\begin{equation}\label{fatpointresolution}
 R\leftarrow S[z_0]\xleftarrow{(f_i)}S[z_0](-2)^{\oplus m}\xleftarrow{L} S[z_0](-3)^{\oplus n_2}\leftarrow \cdots \leftarrow S[z_0](-d)^{\oplus n_r}\leftarrow 0.
\end{equation}
Recall that $f_i(0,z_1,\dots,z_r)=q_i$ and $L_{ij}(0,z_1,\dots,z_r)=l_{ij}$. We extend $q_i$ and $l_{ij}$ by powers of $z_0$ using the following algorithm. 

Consider the following matrices
\begin{itemize}
 \item $q:=(q_1,q_2,\dots,q_m)\in \operatorname{Mat}_{1\times m}(\mathcal{A})$,
 \item $\overline z:=(z_1,\dots,z_r)\in \operatorname{Mat}_{(1\times r)}(\mathcal{A})$,
 \item $C:=[c_{ij}]\in \operatorname{Mat}_{r\times m}(\mathcal{A})$,
 \item $D:=[d_{i}]\in \operatorname{Mat}_{1\times m}(\mathcal{A})$,
 \item $l\in \operatorname{Mat}_{m\times n_2}(\mathcal{A})$, the first syzygy matrix of $\overline{I}$,
 \item $N:=[n_{ij}]\in \operatorname{Mat}_{m\times n_2}(\mathcal{A})$.
\end{itemize}
where $\mathcal{A}=k[z_i,c_{ij},d_i,n_{ij}]$. Notice that $\II_q$ is the set of relations between the variables $c_{ij},d_i$ and $n_{ij}$ that make the following equality true
$$\left(q-\overline z Cz_0-Dz_0^2\right)\left(l+Nz_0\right)=0.$$

Decomposing the equality in powers of $z_0$ we will get a tower of schemes 
$$\BH\left(\overline{R},1\right)=\BH^{(3)}\left(\overline{R},1\right)\rightarrow
\BH^{(2)}\left(\overline{R},1\right)\rightarrow\BH^{(1)}\left(\overline{R},1\right)\rightarrow \BH^{(0)}\left(\overline{R},1\right)=\text{pt},$$
where each $\BH^{(k)}\left(\overline{R},1\right)$ contains the relations given by the coefficients of $z_0^k$.

\begin{enumerate}
 \item  $ql=0$,
 
 which is true by definition of syzygy.
 \item  $z_0\left(qN-\left(\overline z C\right)l\right)=0$.
 
 The entries of this matrix are polynomials of degree $2$ in the $\{z_i\}_{i=1}^r$ and linear in $\{n_{ij}, c_{ij}\}$. As the coefficient $n_{ij}$ multiplies by the polynomial $q_i$, and the $\{q_i\}$ are linear independent, we get linear equations identifying each $n_{ij}$ with a linear combination of the $c_{ij}$. As the $n_{ij}$ can appear multiplying more than a single monomial, we also derive linear relations between the $c_{ij}$.
 \item  $z_0^2\left(\overline z CN+Dl\right)=0$. 
 
 As we can write $N$ with the entries of $C$, this equation gives us identities $d_{i}=h_i(c_{ij})$, where the $h_i$ are quadratic polynomials, and quadratic equations between the $c_{ij}$.
 \item  $z_0^3\left(DN\right)=0$.
 
 This last equation should give us cubic equations in the $c_{ij}$ but we shall see that they are already contained in the ideal generated by the equations found before.
\end{enumerate}

\begin{remark}
By step $(3)$ the $d_i$ are determined by the $c_{ij}$, so we consider $\II_q$ to be the ideal defined by the relations between the $c_{ij}$ as they completely parametrize $\BH(\overline{R},1)$.
\end{remark}

\begin{proposition}\label{quadeq}
Let $q=(q_1,\dots,q_m)$ be a minimal set of homogeneous generators of an ideal with codimension $r\geq 2$ in $k[z_1,\dots,z_r]$ such that for all $i$, $\deg(q_i)=2$, and the first syzygy matrix of the ideal $(q_1,\dots,q_m)$ is linear. Then the ideal of $q$-relations $\II_q$ is generated by quadratic polynomials. 
Furthermore, if $q'$ is obtained from $q$ by a linear change of variables, then $\II_q\cong \II_{q'}$. 
\end{proposition}

\begin{proof}
As $q$ is generated by quadratic polynomials, $\Hom_{\overline R}(q/q^2,\overline{R})_{-3}=0$. Therefore, we just have to prove that
$\varphi_3\colon \HH(\overline{R},1)^{(3)}\rightarrow \HH(\overline{R},1)^{(2)}$ is surjective, which equivalent to prove that $\Ext^1_{\overline{R}}(q/q^2,\overline{R})_{-3}=0$.

The ideal $q$ is CM, hence 
$\Ext^i_{\overline{R}}(q/q^2,\overline{R})\cong\Ext^i_{\overline{R}}(q,\overline{R})
\cong\Ext^i_{\overline{S}}(q,\overline{R})\cong\Ext^i_{\overline{S}}(q,\overline{S})\otimes\overline{R}\neq 0$ 
if and only if $i=0$ or $i=r$. 
Therefore, the result is true for $r\geq 3$. For $r=2$, as $q$ is generated by quadratic polynomials and the syzygy matrix is linear, $q=(z_1^2,z_1z_2,z_2^2)$. This is a triple cover which is a determinantal variety determined by the $2\times 2$ minors of its syzygy matrix, whose entries only satisfy linear relations. Therefore $\II_q$ only contains the quadratic equations defining $d_i$ (see Example (\ref{triplecoverexample})).

The second statement is direct as a linear change of variables on the $z_i$ induces one in the $c_{ij}$.
\end{proof}

\begin{example}\label{triplecoverexample}[Triple Covers]
Let $\overline{z}=(z_1,z_2)$, $q=(z_1^2,z_1z_2,z_2^2)$ and 
$$l=\left(\begin{matrix}
   0 & z_2 \\
   -z_2 & -z_1  \\
   z_1 & 0     
\end{matrix}\right).$$
We want to determine the matrices $C_{2\times 3}$, $N_{3\times 2}$ and $D_{1\times 3}$. Applying a change of variables so that each of the $z_i$ is trace free and using $(qN-\overline{z}Cl)=0$ we get that
$$(q-\overline{z}C)=
(\begin{array}{ccc}
 z_1^2-c_1z_1-c_0z_2 & z_1z_2+c_2z_1+c_1z_2 & z_2^2-c_3z_1-c_2z_2
\end{array}),$$
$$(l+N)=\left(\begin{matrix}
   c_3 & z_2+c_2  \\
   -z_2+2c_2 & -z_1+2c_1  \\
   z_1+c_1 & c_0      
\end{matrix}\right).$$

The left kernel of the matrix $(l+N)$ over $S$ is generated by its $2\times 2$ minors so we get   
$$
D=\left(\begin{array}{ccc}
 2(c_0c_2-c_1^2) & -(c_0c_3-c_1c_2) & 2(c_1c_3-c_2^2)
\end{array}\right).
$$

Computing $D$ using the equation $\overline{z}CN+Dl=0$ would give us the same result, so we find that $\HH(\overline{R},1)=\HH^{(2)}(\overline{R},1)$. This was expected from the work of Miranda, a triple cover is determined by an element of $\Hom(S^3E,\bigwedge^2E)$, which is naturally isomorphic to $\text{TCHom}(S^2E,E)$ (see \cite[Prop. $3.3$]{triple}).
\end{example}

Notice that for $q=(z_iz_j)_{1\leq i\leq j\leq d-1}$ the ideal $\II_q$ gives us the local conditions under which a map $S^2\EE\rightarrow \EE$ induces an associative map $S^2\EE\rightarrow \OO_Y\oplus\EE$, where $\EE$ is a locally free $\OO_Y$-module with rank $d-1$. If we add the linear conditions on the $c_{ij}$ that make the variables $z_i$ trace free, we get the local structure of a general covering map of degree $d$.

The case of covering maps with degree $4$ is worked out in \cite{quadruple} where it is proven that 
$\II_q+(\text{trace free conditions})$
are the equations of an affine cone over $\Gr(2,6)$, the Grassmannian of two dimensional subspaces of a six dimensional space, under its natural Pl\"ucker embedding in $\PP^{14}$. 

Important to mention that Proposition \ref{quadeq}, for quadruple covers, was verified by computing the cubic relations between the $c_{ij}$ (step $(4)$) and noticing that they where in the ideal generated by the quadratic ones. We now have a proof that this is the case for covers of any degree $d$.

\subsection{Global analysis}\label{globalanal}

Recall that for $\EE$, a locally free sheaf of $\OO_Y$-modules of rank $2$, Miranda defines a triple cover homomorphism as a morphism $\phi\colon \Sym^2\EE\rightarrow\EE$ that locally is of the form
\begin{equation}
\begin{array}{rcl}
 \phi(z_1^2) & = & c_1z_1+c_0z_2 \\
 \phi(z_1z_2) & = & -c_2z_1-c_1z_2 \\
 \phi(z_2^2) & = & c_3z_1+c_2z_2.
\end{array}
\end{equation}
In this section we generalise the definition of cover homomorphism. 

\begin{definition}\label{coverhomomorphism}
Given a scheme $Y$ and $\FF$ a locally free $\OO_Y$-module of rank $r$, let $\Q$ be a direct summand of $\RR(\FF)$ such that $\RR(\FF)/\langle\Q\rangle$ is a finitely generated $\OO_Y$-module, $\IN(\RR(\FF)/\langle\Q\rangle)=Q$, and $\Q \cap \OO_Y\oplus \FF=\{0\}$.
Then a morphism $\Phi\in \Hom\left(\Q,\RR(\FF)/(\RR(\FF)\Q)\right)$ is called a \textbf{cover homomorphism} if for every local basis $\{z_i\}$ of $\FF$ and $\{q_j\}$ of $\Q$, the ring
$$\OO_Y[z_i]/\langle q_j-\Phi(q_j)\rangle,$$
is a flat deformation of $\OO_Y[z_i]/\langle q\rangle$. 
If $\FF=\EE$ and $\Q=\Sym^2\EE$ then the basis $\{z_i\}$ is chosen to be a trace free basis. We denote cover homomorphisms by $\CHom\left(\FF;\Q,\RR(\FF)/\langle\Q\rangle\right)$.
\end{definition}

When saying that $\OO_Y[z_i]/\langle q_j-\Phi(q_j)\rangle$ is a flat deformation of $\OO_Y[z_i]/\langle q\rangle$ we consider the polynomials to be locally homogeneous by introducing a variable $z_0$. Geometrically this is just the embedding of $\Spec \left(\Sym^\bullet \FF\right)$ in $\Proj\left( \Sym^\bullet\left(\OO_Y\oplus\FF\right)\right)$.

\begin{theorem}\label{thmglobal}
Let $X$ and $Y$ be schemes, $Y$ locally Noetherian, and $\varphi\colon X\rightarrow Y$ a covering map. Then a commutative and associative $\OO_Y$-algebra structure on $\varphi_*\OO_X$ is equivalent to a pair $(\FF,\Phi)$, where $\FF$ is a sheaf of $\OO_Y$-modules such that $\varphi_*\OO_X$ is a direct summand of $\RR(\FF)$ and $\Phi$ is a covering homomorphism in $\CHom(\FF;\Q,\RR(\FF)/\langle\Q\rangle)$. 

Furthermore, if $\Q$ is locally generated by quadratic polynomials and the ideal defined by them has a linear first syzygy matrix, then $i$ is defined by a morphism in $\Hom(\Q,\FF)$.
\end{theorem}

\begin{proof}
It is direct that a pair $(\FF,\Phi)$ determines a commutative and associative multiplication on $\varphi_*\OO_X$. In the other direction just consider the decomposition $\varphi_*\OO_X=\OO_Y\oplus\EE$. Then a commutative and associative multiplication on $\varphi_*\OO_X$ gives a pair $(\EE,\Phi)$ where $\Phi\in\Hom(\EE;\Sym^2\EE,\OO_Y\oplus\EE)$.

The last sentence is a \mbox{consequence} of Proposition \ref{quadeq}.
\end{proof}

Theorem \ref{thmglobal} is not a global theorem as the theorems presented in \cite{triple,quadruple} for triple and quadruple covers are. Nonetheless, this theorem will allow to explicitly construct the section ring of algebraic covers. An important case we analyze in the next section is the case of Gorenstein covers for which we have a concrete decomposition of $\varphi_*\OO_X$.

%For a locally free $\OO_Y$--module with rank $2$, Miranda defined triple cover homomorphisms $\Phi\in \operatorname{TCHom}\left(\Sym^2\EE,\EE\right)$ as homomorphisms $\Sym^2\EE\rightarrow \EE$ that is locally of the form
%\begin{equation}
%\begin{array}{rcl}
% \phi(z_1^2) & = & c_1z_1+c_0z_2 \\
% \phi(z_1z_2) & = & -c_2z_1-c_1z_2 \\
% \phi(z_2^2) & = & c_3z_1+c_2z_2,
%\end{array}
%\end{equation}
%where $\{z_1,z_2\}$ is a local basis for $\EE$ and $c_i\in \OO_Y$. As Miranda proved, and we have seen in Example %\ref{triplecoverexample}, the (local) form of these morphisms define a homomorphism %$\Hom\left(\Sym^2\EE,\OO_Y\oplus\EE\right)$ that determines a triple cover $X\rightarrow Y$.

In the next Lemma, we prove that the local structure of a cover homomorphism determines the global one.  

\begin{lemma} Let $Y$ be a locally Noetherian irreducible scheme, and $\FF$ a locally free $\OO_Y$-module of rank $r$. Then $\Phi\in\Hom(\Q,\RR(\FF)/\langle\Q\rangle)$ is a cover homomorphism, for $\Q$ a direct summand of $\RR(\FF)$, if and only if there exists a point $y\in Y$ such that the induced homomorphism on the stalk  
$$\Phi_y\colon \Q_y\rightarrow \left(\RR(\FF)/\langle\Q\rangle\right)_y $$ is a cover homomorphism.
\end{lemma}

\begin{proof}
Let $(z_1,\dots,z_r)$ be a basis for $\FF\otimes \OO_{Y,y}$ and $(q_1,\dots,q_m)$ a set of generators for $\Q\otimes \OO_{Y,y}$. As $\Phi$ is a cover homomorphism, the $c_{ij}\in\OO_{Y,y}$ associated with the morphism $\Phi$ satisfy the relations in $\II_q$.

Let $\UU\subset Y$ an open set for which all $c_{ij}\in\OO_Y(\UU)$, then for any $y'\in\UU$, the image of the $c_{ij}$ in $\OO_{Y,y'}$ also satisfy the relations in $\II_q$. 

Given any open set $\UU'\subset Y$, as $Y$ is irreducible and the transition morphisms in $\Q$ and $\RR(\FF)$ are given by $\OO_Y$-linear automorphisms, the $c_{ij}'$, where $\{z'_i\},\{q'_i\}$ are basis for $\FF$ and $\Q$ over $\UU'$, satisfy the relations in $\II_{q}\cong\II_{q'}$.  
\end{proof}

\section{Gorenstein Covers}\label{GorensteinCovers}
A covering map $\varphi\colon X\rightarrow Y$ is called a Gorenstein covering map if all the fibres $X_y$, for $y\in Y$, are Gorenstein, see \cite[Prop.$9.6$]{HartshorneRD}.
Casnati and Ekedahl studied these covering maps and proved the following result.

\begin{theorem}\cite[Theorem $2.1$]{casnatiI}\label{casnatithm} Let $X$ and $Y$ be schemes, $Y$ integral and let $\varphi\colon X\rightarrow Y$ be a Gorenstein cover of degree $d\ge 3$. There exists a unique $\PP^{d-2}$-bundle $\pi\colon \PP\rightarrow Y$ and an embedding $i\colon X\hookrightarrow \PP$ such that $\varphi=\pi\circ i$ and $X_y:=\varphi^{-1}(y)\subseteq{\PP}_y:=\pi^{-1}(y)\cong\PP^{d-2}$ is a non--degenerate arithmetically Gorenstein subscheme for each $y\in Y$.
Moreover the following hold.
\begin{enumerate}
 \item $\PP\cong\PP(\EE)$ where $\EE^\vee\cong \operatorname{coker}\varphi^{\#}$, $\varphi^{\#}\colon \OO_Y\rightarrow\varphi_*\OO_X$.
 \item The composition $\varphi\colon \varphi^*\EE\rightarrow\varphi^*\varphi_*\omega_{X|Y}\rightarrow\omega_{X|Y}$ is surjective and the ramification divisor $R$ satisfies $\OO_X(R)\cong\omega_{X|Y}\cong\OO_X(1):=i^*\OO_{\PP(\EE)}(1)$.
 \item There exists an exact sequence $\mathcal{N}_*$ of locally free $\OO_{\PP}$-sheaves
 \begin{equation*}\label{symrescas}
  0\leftarrow \OO_X\leftarrow\OO_{\PP}\xleftarrow{\alpha_1}\mathcal{N}_1(-2)\xleftarrow{\alpha_2}\cdots \xleftarrow{\alpha_{d-3}}\mathcal{N}_{d-3}(-d+2)\xleftarrow{\alpha_{d-2}}\mathcal{N}_{d-2}(-d)\leftarrow 0
 \end{equation*}
 unique up to isomorphisms and whose restriction to the fibre $\PP_y:=\pi^{-1}(y)$ over $y$ is a minimal free resolution of the structure sheaf of $X_y:=\varphi^{-1}(y)$, in particular $\mathcal{N}_i$ is fibrewise trivial. $\mathcal{N}_{d-2}$ is invertible and, for $i=1,\dots,d-3$, one has 
 $$\operatorname{rk}\mathcal{N}_i=\beta_i=\frac{i(d-2-i)}{d-1}\binom{d}{i+1},$$
 hence $X_y\subset \PP_y$ is an arithmetically Gorenstein subscheme. Moreover $\pi^*\pi_*\mathcal{N}_*\cong \mathcal{N}_*$ and $\HHom\left(\mathcal{N}_*,\mathcal{N}_{d-2}(-d)\right)\cong\mathcal{N}_*$.
 \item If $\PP\cong\PP(\EE')$, then $\EE'\cong\EE$ if and only if $\mathcal{N}_{d-2}\cong \pi^*\det\EE'$ in the resolution (\ref{symrescas}) computed with respect to the polarization $\OO_{\PP(\EE')}(1)$.
\end{enumerate}
\end{theorem}

Definition \ref{coverhomomorphism} of cover homomorphism uses the global structure of $\varphi_*\OO_X$ and not its fibre structure. An example where we can see that it is an important detail is the case of covering maps of degree $3$ without total ramification points. Then any fibre can be described by the vanishing of a single polynomial of degree $3$, so the fibre is Gorenstein, although the global multiplication is defined by three polynomials that are a deformation of a Cohen--Macaulay but not Gorenstein ideal.

In the next Theorem we are interested in Gorenstein covering maps for which the fibre structure is induced by the global structure of $\varphi_*\OO_X$.

\begin{theorem}\label{socle}
 Let $\varphi\colon X\rightarrow Y$ be a covering map of degree $d\geq 4$, and $\LL$ a line bundle on $Y$ such that $\omega_Y\cong \LL^{\otimes k_Y}$ and $\omega_X\cong\varphi^*\LL^{\otimes k_X}$, $k_Y, k_X\in\ZZ$, $k_X\neq k_Y$. Then 
 $$\varphi_*\OO_X=\OO_Y\oplus\FF\oplus \mathcal{L}^{\otimes k_Y-k_X}.$$
 Furthermore, $\varphi_*\OO_X$ is a direct summand of $\RR(\FF)$.
\end{theorem}

Consider $R$, a commutative ring with identity, and $A$ a commutative $R$-algebra, with unit element $e_0$, finitely generated, and projective as an $R$-module. Then the following conditions are equivalent.
\begin{itemize}
    \item As $A$-modules, $A\cong\Hom_R(A,R)$.
    \item There exists a nonsingular $R$-bilinear form $\langle\, ,\,\rangle\colon A\times A\rightarrow R$ such that $\langle ab,c\rangle=\langle a,bc\rangle$, for all $a,b,c\in A$.
    \item There exists a surjective $R$-linear form $\eta\colon A\rightarrow R$, such that $\ker \eta$ contains no ideal other than $(0)$. 
\end{itemize}

If these conditions hold for an algebra $A$, then it is called a Frobenius $R$-algebra.  
In \cite{Behnke} it is proven that, if the form $\eta$ satisfies $\eta(e_0)=0$, then there exists $e^*\in A$ such that 
$$A=R\cdot e_o\oplus\FF\oplus R\cdot e^*$$
and $A=R\cdot e_0+\FF+\FF^2$. The proof of Theorem \ref{socle} is based on their argument.

\begin{proof}
By adjunction formula
$$\varphi_*\omega_X=\HHom_{\OO_Y}(\varphi_*\OO_X,\omega_Y)\cong\HHom_{\OO_Y}(\varphi_*\OO_X,\OO_Y)\otimes \mathcal{L}^{\otimes k_Y}.$$
On the other hand, by assumption,
$$\varphi_*\omega_X\cong\varphi_*\varphi^*\LL^{\otimes k_X}=\varphi_*\OO_X\otimes \LL^{\otimes k_X}.$$
We conclude that, as $\OO_X$-modules, $\varphi_*\OO_X$ and $\HHom_{\OO_Y}(\varphi_*\OO_X,\OO_Y)$ are isomorphic. Denote this isomorphism by  $\Psi\colon \varphi_*\OO_X\rightarrow\HHom_{\OO_Y}(\varphi_*\OO_X,\OO_Y)\otimes \MM$, where $\MM=
\LL^{k_Y-k_X}$.

Recall that from the trace map splitting we have the following short exact sequence
$$0\rightarrow \OO_Y\xrightarrow{\varphi^\#}\varphi_*\OO_X \rightarrow \EE\rightarrow 0.$$
Applying $\Psi$ to the middle term and tensoring with $\MM^{-1}$ we have 
$$
\begin{array}{cl}
     &  0\rightarrow \MM^{-1}\xrightarrow{\Psi\circ\varphi^\#}\HHom_{\OO_Y}(\varphi_*\OO_X,\OO_Y) \rightarrow \EE\otimes\MM^{-1}\rightarrow 0 \\
\Leftrightarrow & 0\rightarrow \MM^{-1}\xrightarrow{\Psi\circ\varphi^\#}\OO_Y\oplus\EE^\vee \rightarrow \EE\otimes\MM^{-1}\rightarrow 0.
\end{array}
$$
This sequence splits as the composition of $\Psi\circ\varphi^\#$ with $\tr \circ \Psi^{-1}$ is the identity. In particular, as $\MM\neq \OO_Y$ (because $k_Y\neq k_Y$), $\MM^{-1}$ is a direct summand of $\HHom_{\OO_Y}(\varphi_*\OO_X,\OO_Y)$ and we have the decomposition 
 $$\varphi_*\OO_X=\OO_Y\oplus\FF\oplus \MM.$$

Notice that $\varphi^\#(1_Y)$ is the generator of $\varphi_*\OO_X$ as an $\OO_X$-module, hence, $\Psi(\varphi^\#(1_Y))$ is a generator of $\HHom(\varphi_*\OO_X,\OO_Y)$. For simplicity we will use the notation $1_X=\varphi^\#(1_Y)$ and $\eta=\Psi(\varphi^\#(1_Y))$.

One can use $\Psi$ to define an $\OO_Y$-bilinear form $\langle\, ,\,\rangle\colon \varphi_*\OO_X\times\varphi_*\OO_X\rightarrow \OO_Y$ by 
$$\langle a,b\rangle = \Psi(a)(b)=a\Psi(\varphi^\#(1_Y))(b)=a\eta(b)=\eta(ab).$$
It is a nonsingular form ($a\eta=0\Rightarrow a=0$), and $\eta\colon \varphi_*\OO_X\rightarrow\OO_Y$ is surjective with no ideal other than the zero ideal in its kernel (if $\eta(\II)=0$, for a sheaf of $\OO_X$-ideals $\II$, and $0\neq a\in \II$, then $\eta(ab)=0$ for all $b\in\varphi_*\OO_X$, which is a contradiction).

Mind that in $\HHom(\varphi_*\OO_X,\OO_Y)=\OO_Y\oplus\FF^\vee\oplus\MM^\vee$, $\eta$ is the summand $\MM^\vee$. In particular, $\OO_Y\oplus\FF$ is the kernel of $\eta$, and so $\eta(1_X)=0$.
As $\HHom_{\OO_Y}(\varphi_*\OO_X,\OO_Y)$ is generated by $\eta$, let $e\in \OO_X$ such that $e\cdot\eta=\tr$, then the following equalities hold
$$1_Y=e\cdot\eta(1_X)=\eta(e\cdot 1_X)=\eta(e).$$

The component $\FF$ of $\varphi_*\OO_X$ is the orthogonal complement of $1_X$ and $e$, with respect to the bilinear form on $\varphi_*\OO_X$, as for $z\in\FF|_\UU$ 
$$\langle 1_X,z\rangle =\eta(z)=0=\tr(z)=\langle e,z\rangle.$$

Furthermore, when restrited to $\FF$, the form $\langle\, ,\,\rangle$ is nonsigular. If there is an element $f\in\FF|_\UU$ such that $\langle f,g\rangle=0$ for all $g\in \FF|_\UU$, then $\langle f,x\rangle=0$ for all $x\in \varphi_*\OO_X|_\UU$, but the bilinear form is nonsingular by construction. 

Moreover, the form $\eta|_{\Sym^2\FF}$ is surjective. Realise that $\operatorname{Im}\left(\eta|_{\Sym^2\FF}\right)$ is closed by multiplication by $\OO_Y$, so assume it is a sheaf of $\OO_Y$-ideals $\II$. Over an open set such that $\OO_Y(\UU)$ is affine, consider a maximal ideal $y\subset \OO_Y(\UU)$ containing $\II|_\UU$. Then we get a contradiction as $\Psi\otimes k(y)$ is an isomorphism of vector spaces and all the statements written in the proof are still valid for the covering map $\varphi_y$. In particular, $\langle\, ,\,\rangle$ is nonsingular when restricted to $\FF_y$, but $\operatorname{Im}(\eta|_{\Sym^2\FF_y})\subset \II_\UU\otimes k(y)=0$.

We conclude that $\varphi_*\OO_X=\OO_Y+\FF+\Sym^2\FF$ and therefore $\varphi_*\OO_X$ is a direct summand of $\RR(\FF)$.
\end{proof}

Notice that Theorem \ref{socle} proves that for a covering map $\varphi\colon X\rightarrow Y$ and $\LL$ a line bundle on $Y$ in the conditions of imposed on the Theorem, $\varphi$ is determined by a cover homomorphism $\Phi\in\CHom(\FF;\Q,\OO_Y\oplus\FF\oplus \LL^{k_Y-k_X})$, where $\Q$ is the kernel of the map $S^2\FF\rightarrow \LL^{\otimes k_X-k_Y}$.

Furthermore, $\omega_{X}\otimes k(y)$ is generated by $e\otimes k(y)$, for every point $y\in Y$. In particular every fibre $X_y$ is Gorenstein (see \cite[Prop. $21.5$]{eisenbudCA}), so the structure Theorem \ref{casnatithm} applies and one can rewrite similar results with the embedding in $\AA^{d-2}(\FF^\vee)$ instead of $\PP^{d-2}(\FF^\vee\oplus\LL^{k_X-k_Y})$.
 
This might not be an advantage in practicall terms as, although $\FF$ is determined, $\Q$ may not be. As example take the surfaces $S_d$ such that $K(S_d)=d, p_g(S_d)=3, q(S_d)=0$, $d\geq 2$. In \cite{casII}, Casnati proves that the canonical map of $S_d$ induces a Gorenstein cover of degree $d$, $\varphi\colon X_d\rightarrow\PP^2$, such that 
$$\varphi_*\OO_{X_d}=\OO_{\PP^2}\oplus\OO_{\PP^2}(-2)^{\oplus(d-2)}\oplus\OO_{\PP^2}(-4).$$
As $\Sym^2\left(\OO_{\PP^2}(-2)^{\oplus(d-2)}\right)$ is generated by $ \binom{d-1}{2}$ copies of $\OO_{\PP^2}(-4)$, there are different possible choices for $\Q$. As the fibres of $\varphi$ are Gorenstein ideals of codimension $d-2$, one can only completely describe the surfaces $S_d$ for $d\leq 5$, i.e. up to codimension $3$.
In the next Section we will construct a model for Gorenstein covers of degree $6$ under extra assumptions on $\FF$.

\subsection{Degree $6$ covering map}\label{Codimension4}

Let us quickly recall the definition of $\oGr(5,10)$, the orthogonal Grassmann variety and the equations defining its embedding in $\PP^{15}$. One can read full descriptions in \cite{ogr,MukaiOGr}. 

Consider $V=\CC^{10}$ a vector space with a nondegenerate quadratic form $q$ so that $V=U\oplus U^\vee$. A subspace of $V$ is called isotropic if $q$ is identically zero on it, e.g. $U$ and $U^\vee$ are isotropic.

The subspace of isotropic $5$-dimensional subspaces of $V$ has two components that can be divided by the parity of the intersection with $U$, i.e. if $F$ is a continuous family of generators of maximal isotropic subspaces of $V$, then $(F\cap U \mod 2)$ is locally constant. We pick the component that contains $U$, and so define the orthogonal Grassman variety $\oGr(5,10)$ as 
$$\operatorname{OGr}(5,10)=\left\{\begin{array}{rcl} F \in \operatorname{Gr}(5,V) & \bigg| & \begin{array}{c} \text{F is isotropic for }p \\ \text{and }\dim F\cap U \text{ is odd}  \end{array} \end{array}\right\}.$$

This Grassmannian embeds into $\PP\left(\CC\oplus\bigwedge^2 U\oplus\bigwedge^4 U\right)$ as a homogeneous space for the special orthogonal group $\operatorname{SO}_\CC(10)$. It is defined by the vanishing of the following polynomials 
$$
\begin{array}{rcl}
    N_1 & = &  \xi_{0}\xi_{2345} - \xi_{23}\xi_{45} + \xi_{24}\xi_{35} - \xi_{25}\xi_{34},  \\
    N_{-1} & = & \xi_{12}\xi_{1345} - \xi_{13}\xi_{1245} + \xi_{14}\xi_{1235} - \xi_{15}\xi_{1234}, \\
    N_2 & = & \xi_{0}\xi_{1345} - \xi_{13}\xi_{45} + \xi_{14}\xi_{35} - \xi_{15}\xi_{34}, \\
    N_{-2} & = &  \xi_{12}\xi_{2345} - \xi_{23}\xi_{1245} + \xi_{24}\xi_{1235} - \xi_{25}\xi_{1234},  \\
    N_3 & = &  \xi_{0}\xi_{1245} - \xi_{12}\xi_{45} + \xi_{14}\xi_{25} - \xi_{15}\xi_{24}, \\
    N_{-3} & = &  \xi_{13}\xi_{2345} - \xi_{23}\xi_{1345} + \xi_{34}\xi_{1235} - \xi_{35}\xi_{1234}, \\
    N_4 & = & \xi_{0}\xi_{1235} - \xi_{12}\xi_{35} + \xi_{13}\xi_{25} - \xi_{15}\xi_{23}, \\
    N_{-4} & = & \xi_{14}\xi_{2345} -\xi_{24}\xi_{1345} + \xi_{34}\xi_{1245} - \xi_{45}\xi_{1234}, \\
    N_5 & = & \xi_{0}\xi_{1234} - \xi_{12}\xi_{34} + \xi_{13}\xi_{24} - \xi_{14}\xi_{23},  \\
    N_{-5} & = & \xi_{15}\xi_{2345} - \xi_{25}\xi_{1345} + \xi_{35}\xi_{1245} - \xi_{45}\xi_{1235}. 
\end{array}
$$
Where the variable $\xi_0$ corresponds to the term $\CC$ and, denoting by $\langle f_1,\dots,f_5\rangle$ a basis of $U$, the $\xi_{ij}$ correspond to $(f_i\wedge f_j)$, a basis of $\bigwedge^2 U$, and $\xi_{ijkl}$ to $(f_i\wedge f_j\wedge f_k\wedge f_l)$, basis of $\bigwedge^4 U$. 

\begin{theorem}\label{maintheorem4}
 Let $\varphi\colon X\rightarrow Y$ be a Gorenstein covering map such that 
 $$\varphi_*\OO_X=\OO_Y\oplus M\oplus M\oplus\wedge^2M,$$
 where $M$ is a simple $\OO_Y$-module with rank $2$. Then 
 \begin{enumerate}
  \item $\varphi$ determines and is determined by a morphism in $\Hom\left((S^2M)^{\oplus 3},M^{\oplus 2}\right)$; 
  \item if $\{z_1,z_2,w_1,w_2\}$ is a local basis for $M\oplus M$ (where $w_i$ is the image of $z_i$ by an isomorphism from $M$ to $M$), we get that $\IN(\II_X)$ is locally generated by the polynomials 
  \begin{equation}\label{q}
q=\left(
\begin{array}{ccccccccc}
 z_1^2 & z_1z_2 & z_2^2 &z_1w_1  & \frac{1}{2}(z_1w_2+z_2w_1) & z_2w_2 & w_1^2 & w_1w_2 & w_2^2
\end{array}\right);
\end{equation}
  \item the ideal $\II_q$ determining the local structure of a fibre of $\varphi$ is given by the polynomials defining the spinor embedding of the orthogonal Grassmannian $\oGr(5,10)$ in $\PP^{15}$;
  \item $\varphi$ is a deformation of an $S_3$-Galois branched cover (which corresponds to a linear section of the embedding of $\oGr(5,10)$ in $\PP^{15}$).  
 \end{enumerate}
\end{theorem}

\begin{proof} By Theorem \ref{socle} and the remark made after, $\varphi$ is determined by an element of 
$$\CHom(M^{\oplus 2};\Q,\varphi_*\OO_X),$$ 
where $\Q$ is a direct summand of $S^2(M\oplus M)$, that by Proposition \ref{quadeq}, is determined by a morphism in $\Hom\left(Q,M^{\oplus 2}\right)$. As $M$ is simple
$$\begin{array}{ccl}
\Hom(M,M)=k & \Leftrightarrow & \Hom(M\otimes M^\vee,\OO_Y) = k \\
            & \Leftrightarrow & \Hom(M\otimes M^\vee\otimes \wedge^2 M,\wedge^2 M) = k  \\
            & \Leftrightarrow & \Hom(M\otimes  M,\wedge^2 M) = k,
 \end{array}$$
hence $(\Sym^2M)^{\oplus 3}$ is the kernel of the morphism $S^2\left(M\oplus M\right)\rightarrow \wedge^2 M$, proving $(1)$. $(2)$ is a consequence of $(1)$ as $q$ is a basis for $(S^2M)^{\oplus 3}$.

Using $q$, we run the algorithm described in Section \ref{iqrelations} (see Appendix) and get
$$C^t=
\left(\begin{array}{rrrr}
c_{32} + 2 c_{43} & - c_{33} & - c_{13} & c_{03} \\
c_{53} & c_{32} & - c_{23} & c_{13} \\
- c_{52} & 2 c_{42} + c_{53} & c_{22} & c_{23} \\
c_{73} & - c_{63} & c_{32} & c_{33} \\
-\frac{1}{2} c_{72} + \frac{1}{2} c_{83} & \frac{1}{2} c_{62} - \frac{1}{2} c_{73} & c_{42} & c_{43} \\
- c_{82} & c_{72} & c_{52} & c_{53} \\
- c_{71} & c_{61} & c_{62} & c_{63} \\
- c_{81} & c_{71} & c_{72} & c_{73} \\
c_{80} & c_{81} & c_{82} & c_{83}
\end{array}\right).$$
With a change of variables 
$$\left(\begin{matrix}
   z_0 \\
   z_1 \\
   w_0 \\
   w_1
  \end{matrix}\right)\mapsto
\left(\begin{matrix}
 z_0 - (c_{32} + c_{43}) \\
 z_1 - (c_{42} + c_{53}) \\
 w_0 - (c_{62} + c_{73}) \\
 w_1 - (c_{72} + c_{83})
\end{matrix}\right),$$
which glues as we are choosing trace-free basis for each component $M$ with respect to the component itself, we get the matrix
$$C^t=
\left(\begin{array}{rrrr}
c_{43} & - c_{33} & - c_{13} & c_{03} \\
c_{53} & - c_{43} & - c_{23} & c_{13} \\
- c_{52} & - c_{53} & c_{22} & c_{23} \\
c_{73} & - c_{63} & - c_{43} & c_{33} \\
c_{83} & - c_{73} & - c_{53} & c_{43} \\
- c_{82} & - c_{83} & c_{52} & c_{53} \\
- c_{71} & c_{61} & - c_{73} & c_{63} \\
- c_{81} & c_{71} & - c_{83} & c_{73} \\
c_{80} & c_{81} & c_{82} & c_{83}
\end{array}\right)\cong
\left(\begin{array}{rrrr}
c_{11} &  c_{10} &  c_{01} & c_{00} \\
-c_{12} & - c_{11} & - c_{02} & -c_{01} \\
 c_{13} &  c_{12} & c_{03} & c_{02} \\
 -c_{21} & -c_{20} & -c_{11} & - c_{10}\\
  c_{22} &  c_{21} & c_{12} &  c_{11} \\
 -c_{23} & -c_{22} & - c_{13} &  -c_{12}\\
 c_{31} & c_{30} &  c_{21} & c_{20} \\
- c_{32} & -c_{31} & - c_{22} & - c_{21} \\
c_{33} & c_{32} & c_{23} & c_{22}
\end{array}\right).$$

In the last matrix we renamed the variables to have a better look at the structure of $C$ and we see that $C$ has the same decomposition as a triple cover homomorphism where each of the entries is a triple cover homomorphism,
\begin{equation}\label{C}
 C^t=\left(
\begin{matrix}
 C_1 & C_0 \\
 -C_2 & -C_1 \\
 C_3 & C_2
\end{matrix}\right),\hspace{2mm}
C_i=\left(
\begin{matrix}
 c_{i1} & c_{i0} \\
 -c_{i2} & -c_{i1} \\
 c_{i3} & c_{i2}
\end{matrix}\right).
\end{equation}

With the second step we get the vector $D$
$$D^t=\left(\begin{array}{l}
-2c_{11}^2 + 2c_{10}c_{12} + 2c_{01}c_{21} - c_{02}c_{20} - c_{00}c_{22} \\ 
-c_{10}c_{13} + c_{11}c_{12} - 2c_{02}c_{21} + c_{03}c_{20} + c_{01}c_{22} \\ 
2c_{11}c_{13} - 2c_{12}^2 - c_{03}c_{21} - c_{01}c_{23} + 2c_{02}c_{22} \\ 
-c_{01}c_{31} + c_{00}c_{32} + c_{11}c_{21} + c_{12}c_{20} - 2c_{10}c_{22} \\ 
\tfrac{1}{2}( -c_{00}c_{33} + c_{01}c_{32} - 5c_{12}c_{21} + c_{13}c_{20} + 4c_{11}c_{22}) \\ 
c_{01}c_{33} - c_{02}c_{32} + c_{13}c_{21} - 2c_{11}c_{23} + c_{12}c_{22} \\ 
2c_{11}c_{31} - c_{12}c_{30} - c_{10}c_{32} - 2c_{21}^2 + 2c_{20}c_{22} \\ 
c_{12}c_{31} + c_{10}c_{33} - 2c_{11}c_{32} - c_{20}c_{23} + c_{21}c_{22} \\ 
-c_{13}c_{31} - c_{11}c_{33} + 2c_{12}c_{32} + 2c_{21}c_{23} - 2c_{22}^2 
\end{array}\right)
$$
and all the quadratic relations that the $c_{ij}$ need to satisfy
$$\II_q=\left(\begin{array}{l}
c_{00}c_{13} - 3c_{01}c_{12} + 3c_{02}c_{11} - c_{03}c_{10} \\ 
c_{00}c_{23} - 3c_{01}c_{22} + 3c_{02}c_{21} - c_{03}c_{20}  \\ 
c_{10}c_{33} - 3c_{11}c_{32} + 3c_{12}c_{31} - c_{13}c_{30} \\ 
c_{20}c_{33} - 3c_{21}c_{32} + 3c_{22}c_{31} - c_{23}c_{30} \\ 
c_{00}c_{31} - c_{01}c_{30} - 3c_{10}c_{21} + 3c_{11}c_{20} \\ 
c_{00}c_{32} - c_{02}c_{30} - 3c_{10}c_{22} + 3c_{12}c_{20} \\ 
c_{00}c_{33} - c_{03}c_{30} - 9c_{11}c_{22} + 9c_{12}c_{21}  \\ 
c_{01}c_{33} - c_{03}c_{31} - 3c_{11}c_{23} + 3c_{13}c_{21} \\ 
c_{02}c_{33} - c_{03}c_{32} - 3c_{12}c_{23} + 3c_{13}c_{22} \\ 
c_{01}c_{32} - c_{02}c_{31} - c_{10}c_{23} + c_{13}c_{20}   \\  
\end{array}\right).
$$

One can rearrange the polynomials defining the embedding of the $\oGr(5,10)$ in $\PP^{15}$ as 
$$
\left\{\begin{array}{l}
    \xi_0 v-\operatorname{Pfaff}(M)=0 \\
    Mv=0,       
\end{array}\right.
$$
where 
$$M=\left(\begin{matrix}
      \xi_{12} & \xi_{13} & \xi_{14} & \xi_{15} \\
             & \xi_{23} & \xi_{24} & \xi_{25} \\
             &        & \xi_{34} & \xi_{35} \\
      -\operatorname{sym}       &        &        & \xi_{45}
\end{matrix}\right), \hspace{2mm} v=
\left(\begin{matrix}
 -\xi_{2345} \\
 \xi_{1345} \\
 -\xi_{1245} \\
 \xi_{1235} \\
 -\xi_{1234}
\end{matrix}\right).
$$

In the case of $\II_q$, take $\xi_0=3c_{11}$ and
$$M=
\left(\begin{matrix}
  3c_{21} & c_{30} & -c_{33} & -c_{31} \\
      & -c_{00} & c_{03} & c_{01}     \\
       &  & 3c_{12} & -c_{10}   \\
     \operatorname{-sym}  &   &  & c_{13}  \\  
\end{matrix}\right),\hspace{2mm}
v= \left(\begin{matrix}
 c_{02} \\ c_{32} \\ c_{23} \\ c_{20} \\ 3c_{22}
\end{matrix}\right).$$

Taking each $C_i=e_i\widetilde C$, where $e_i\in k$ and $\widetilde C$ is a triple cover block
\begin{equation*}
 \widetilde C=\left(\begin{array}{cc}
 c_1 & c_0 \\
 -c_2 & -c_1 \\
 c_3 & c_2
 \end{array}\right),
\end{equation*}
all the relations in $\II_q$ are satisfied (notice the index homogeneity of each polynomial defining $\II_q$). This is a linear component of $\oGr(5,10)$. In Theorem \ref{galoisthm} we prove it is an $S_3$-Galois branched cover. As $\oGr(5,10)$ is connected, statement $(4)$ holds. 
\end{proof}

\subsection{The linear component - $S_3$-Galois branched cover}\label{linearcomponent}

In this section we study a component of the degree $6$ covering maps described above, the component whose local equations correspond to a linear section of $\oGr(5,10)$ given by $C_i=e_i \widetilde C$, for $e_i\in k$ and $\widetilde C$ a `triple cover block'
\begin{equation*}
 \widetilde C=\left(\begin{array}{cc}
 c_1 & c_0 \\
 -c_2 & -c_1 \\
 c_3 & c_2
 \end{array}\right).
\end{equation*}

One can write such covering homomorphism as
\begin{equation}\label{linearequations}
 \begin{array}{rcl}
  \Phi\left(S^2(Z)\right) & = & e_1 \widetilde C Z +e_0 \widetilde C W \\
  \Phi\left(S^2(Z,W)\right) & = & -e_2\widetilde C Z -e_1\widetilde C W \\
  \Phi\left(S^2(W)\right) & = & e_3\widetilde C Z +e_2\widetilde C W, \\
 \end{array}
\end{equation}
where, by abuse of notation, $Z=\left(\begin{matrix}z_1 & z_2\end{matrix}\right)^t, W=\left(\begin{matrix}w_1 & w_2\end{matrix}\right)^t$ and $S^2(Z), S^2(Z,W), S^2(W)$ are vectors with the second symmetric powers of $z_i$ and $w_i$ as entries.

This allows one to think of $Z$ and $W$ as a basis for $\AA^2_Y$ and the cover homomorphism above as a cover homomorphism defining three points over a field. In particular, for $g\in \GL(2,k)$, one has the following action on a local basis 
$$
\left(\begin{array}{cc}z_i & w_i\end{array}\right)\mapsto g\left(\begin{array}{cc}z_i & w_i\end{array}\right)\text{, }\ i\in\{1,2\}.
$$

In the next proposition we study such actions.

\begin{proposition}\label{triplepoints}
 Let $\{q_1,q_2,q_3\} \subset \AA^2=\Spec\left(k[z,w]\right)$ be three points defined by the vanishing of the following polynomials
 $$\left(\begin{array}{c}
  z^2-e_1z-e_0w +2(e_0e_2-e_1^2) \\ zw+e_2z+e_1w-(e_0e_3-e_1e_2) \\ w^2-e_3z-e_2w+2(e_1e_3-e_2^2)
 \end{array}\right),$$
where $e_i\in k$. Then the following hold
\begin{enumerate}
  \item\label{points0} the $q_i$ are pairwise distinct if and only if $\Delta_{tc}(e_0,e_1,e_2,e_3)\neq 0$, where
  $$\Delta_{tc}(e_0,e_1,e_2,e_3):=
 e_0^2e_3^2 + 4e_0e_2^3 - 3e_1^2e_2^2 + 4e_1^3e_3 - 6e_0e_1e_2e_3;
 $$
  \item\label{points1} $\sum_i z(q_i)=\sum_i w(q_i)=0$, i.e. the origin is the barycenter of the points;
  \item\label{points2} via a linear change of variables on the basis $\{z,w\}$ we can change any quadruple $(e_i)_{0\leq i\leq 3}$ defining three distinct points into any other for which $\Delta_{tc}(e_i)\neq 0$;
  \item\label{points3} each quadruple $(e_0,e_1,e_2,e_3)$, for which $\Delta_{tc}(e_i)\neq 0$, is fixed by a representation of $S_3$, the symmetric group of order $3$, in $\GL(2,k)$. 
 \end{enumerate}
\end{proposition}

Notice that we can consider $X=\{q_1,q_2,q_3\}$ to be a triple cover of $\Spec(k)$. Then $(\ref{points0})$ is just stating that the (local) equation defining the branch locus is given by the vanishing of the polynomial $\Delta_{tc}(e_i)$, where $e_i \in k$ define the local equations of a triple cover.
It was already proved in \cite[Lemma $4.5$]{triple} but we show a different proof.

\begin{proof} 
$(\ref{points0})\colon$
With a change of variables $z\mapsto z/t$, $w\mapsto w/t$, and getting rid of denominators, we get three homogeneous polynomials of degree $2$ and hence an embedding of $\{q_1,q_2,q_3\}$ in $\PP^2_{\langle z,w,t\rangle}$. We can then project from the point $(0,0,1)$, which corresponds to the elimination of the variable $t$ from the ideal, and obtain $\{q_1,q_2,q_3\}$ in $\PP^1$ defined by the vanishing of the polynomial 
\begin{equation*}\label{triplecoverp1}
\begin{matrix}
(e_1e_2e_3-\tfrac{1}{3}e_0e_3^2-\tfrac{2}{3}e_2^3)z^3+(2e_1^2e_3-e_1e_2^2-e_0e_2e_3)z^2w \\ 
\hspace{20mm} +(e_1^2e_2+e_0e_1e_3-2e_0e_2^2)zw^2 +(\tfrac{2}{3}e_1^3-e_0e_1e_2+\tfrac{1}{3}e_0^2e_3)w^3.
\end{matrix}
\end{equation*}

The discriminant of a cubic polynomial $p(x)=ax^{3}+bx^{2}+cx+d$ is 
$$b^{2}c^{2}-4ac^{3}-4b^{3}d-27a^{2}d^{2}+18abcd,$$ (see \cite[Lemma $1.19$, Example $1.22$]{MukaiInvMod}), hence we get the result by direct computation. 

$(\ref{points1})\colon$ Assume that the points are not collinear. Applying a linear change of variables on $(z,w)$, we can assume that $q_1,q_2$ are vertex points, i.e. $z(q_1)=w(q_2)=1$ and $z(q_2)=w(q_1)=0$.

Evaluating the polynomial $z^2-e_1z-e_0w +2(e_0e_2-e_1^2)$ on the three points 
\begin{equation*}
 \left\{\begin{array}{lcl}
  0 & = & 1 - e_1 + 2(e_0e_2-e_1^2)  \\
  0 & = & - e_0 + 2(e_0e_2-e_1^2)  \\
  0 & = & z^2(q_3)-e_1z(q_3)-e_0w(q_3) + 2(e_0e_2-e_1^2)
 \end{array}\right.
\end{equation*}
one concludes that $e_1=1+\frac{z^2(q_3)-z(q_3)}{z(q_3)+w(q_3)-1}$. Notice that $z(q_3)+w(q_3)\neq 1$ as equality would imply collinearity of the points. From the polynomial $zw+e_2z+e_1w-(e_0e_3-e_1e_2)$ we get the following equations
\begin{equation*}
 \left\{\begin{array}{lcl}
  0 & = & e_2 - (e_0e_3-e_1e_2)  \\
  0 & = & e_1 - (e_0e_3-e_1e_2)  \\
  0 & = & zw(q_3)+e_2z(q_3)+e_1w(q_3)-(e_0e_3-e_1e_2),
 \end{array}\right.
\end{equation*}
which implies that $e_1=-\frac{zw(q_3)}{z(q_3)+w(q_3)-1}$ and therefore, $z(q_3)=-1$. By the same argument $w(q_3)=-1$ which implies that $\sum_i z(q_i)=\sum_i w(q_i)=0$. This still holds after any change of variables. We are left with the case of three collinear points.

If $q_1,q_2,q_3$ are distinct but lie in a line $l$, by a linear change of variables, one can assume that $l=\{z=0\}$. Then, as $z(q_i)=0$ for all $i$, the equation $w^2-e_2w+2(e_1e_3-e_2^2)=0$ has only two solutions which is a contradiction. We conclude that the points are collinear if and only if they are in the ramification locus.

As the points in the ramification locus can be obtained as the image of a linear map of three distinct points in $\AA^2$, and a $\GL(2,k)$ map on $(z,w)$ gives the same linear transformation of the pair $\left(\sum_i z(q_i),\sum_i w(q_i)\right)$, this pair is always $(0,0)$.

$(\ref{points2})\colon$ 
We proved above that for each quadruple $(e_i)\notin \Delta_{tc}$, the polynomials vanish in three non-collinear points with the origin as their barycenter, i.e. these points are defined by any two of them. 
With a linear change of basis one can send them to any points at will, which proves the result as three points correspond to a single quadruple $(e_i)$.

$(\ref{points3})\colon$ Using $(\ref{points2})$, take $(e_0,e_1,e_2,e_3)=(1,0,0,1)$. Then $q_1=(1,1)$, $q_2=(\epsilon,\epsilon^2)$ and $q_3=(\epsilon^2,\epsilon)$, where $\epsilon$ is a cubic root of unity. The action of $S_3$ on these three points is generated by a rotation $r$ sending $q_i$ into $q_{i+1}$ and a reflection $\iota$ changing $q_2$ with $q_3$. The representation of $S_3$ on $k^2$ is given by
$$
r\mapsto\left(\begin{array}{rr}
    \epsilon & 0 \\
    0 & \epsilon^2 
  \end{array}\right),\ 
\iota\mapsto\left(\begin{array}{rr}
    0 & 1 \\
    1 & 0 
  \end{array}\right).
$$
\end{proof}

We can now present our last result. In the Theorem we will not assume the sheaf of $\OO_Y$-modules $M$ to be simple. On the other hand we assume that the cover homomorphism defining the cover is in $\CHom\left((S^2M)^{\oplus 3},M^{\oplus 2}\right)$, allowing the model to be used for different cases. Of particular interest is the case of Galois closure of non-Galois triple covers. In \cite[$\S 5.3$]{DihedralCatPer} it is proven that these are $S_3$-covers. The next Theorem describes their local structure.

\begin{theorem}\label{galoisthm}
 Let $\varphi\colon X\rightarrow Y$ be a Gorenstein covering map of degree $6$ such that $$\varphi_*\OO_X=\OO_Y\oplus M\oplus M\oplus \wedge^2 M.$$ 
 Assume that $\varphi$ is defined by a general linear covering homomorphism 
 $$\Phi\in \CHom\left((S^2M)^{\oplus 3},M^{\oplus 2}\right),$$
 i.e. for a local choice of basis $\{z_i,w_i\}_{i=1,2}$ it can be written as in Equations $(\ref{linearequations})$ with $(e_i)\in k^4\backslash \{\Delta(e_i)_{tc}=0\}$. Then $X$ is an $S_3$-Galois branched cover. Furthermore,
 \begin{enumerate}
  \item\label{galois1} with a linear change of coordinates $I_X$ can locally be written in the form
  $$
   \left(\begin{array}{l}
  S^2(Z) - \widetilde C W \\
  S^2(Z,W) -\widetilde D\\
  S^2(W) -\widetilde C Z  \\
 \end{array}\right),$$
with $$
  \widetilde C = \left(\begin{array}{cc}
  c_1 & c_0 \\
  -c_2 & -c_1 \\
  c_3 & c_2 \\
 \end{array}\right),\hspace{2mm}
 \widetilde D = \left(\begin{array}{c}
2(-c_0c_2+c_1^2) \\ c_0c_3-c_1c_2 \\ 2(-c_1c_3+c_2^2)
 \end{array}\right).
  $$
  \item\label{galois2} $X/\ZZ_2$ is a triple cover $\varphi_2\colon X/\ZZ_2\rightarrow Y$ such that $\varphi_{2*}\OO_{X/\ZZ_2}=\OO_Y\oplus M$; 
  \item\label{galois3} $X/\ZZ_3$ is a double cover $\varphi_3\colon X/\ZZ_3\rightarrow Y$ such that $\varphi_{3*}\OO_{X/\ZZ_3}=\OO_Y\oplus \wedge^2 M$.
 \end{enumerate}
\end{theorem}

\begin{proof} By Proposition \ref{triplepoints} - (\ref{points3}), the $S_3$ action is just the action fixing the quadruple $(e_i)$. Using Proposition \ref{triplepoints} - \ref{2}, we can set $(e_i)$ to be $(1,0,0,1)$ proving (\ref{galois1}).
 
 To prove (\ref{galois2}) and (\ref{galois3}), just notice that $(1,0,0,1)$ is fixed by the actions of $r$ and $\iota$ presented in the proof of Proposition \ref{triplepoints} - (\ref{points3}). Locally, the submodule of $\varphi_*\OO_X$ left invariant by the action of $\iota$ is generated by
 $\{1,z_1+w_1,z_2+w_2\}$. Hence the ideal defining $\OO_{X/\ZZ_2}$ is locally defined by the vanishing of the polynomials  
 $$S^2(Z)+2S^2(Z,W)+S^2(W)-\widetilde C(Z+W)-D.$$ 
 The submodule left invariant by the action of $r$ is generated by $\{1,z_1w_2-z_2w_1\}$. The ideal of $\OO_{X/\ZZ_3}$ is locally generated by the vanishing of the polynomial
 \begin{equation*}\begin{split}
& \left(\frac{z_1w_2-z_2w_1}{2}\right)^2 - \left((c_0c_3-c_1c_2)^2 - 4(-c_1c_3+c_2^2)(-c_0c_2+c_1^2)\right)  \\
\Leftrightarrow & \left(\frac{z_1w_2-z_2w_1}{2}\right)^2 - (c_0^2c_3^2 + 4c_0c_2^3 - 3c_1^2c_2^2 + 4c_1^3c_3 - 6c_0c_1c_2c_3). 
\end{split}
\end{equation*}
Notice that $(z_1w_2-z_2w_1)$ is a local generator of $\wedge^2M$. Furthermore, as all these actions are independent of the choice of basis, we conclude
 $$(\varphi_*\OO_X)^{\ZZ_2}\cong\OO_Y\oplus M\text{ and } (\varphi_*\OO_X)^{\ZZ_3}\cong\OO_Y\oplus \wedge^2 M.$$ 
\end{proof}

\section{Appendix}\label{append}
\begin{verbatim}
S.<z0, z1, w0, w1, nij, di, cij>=QQ[]
from sage.libs.singular.function_factory import singular_function
minbase = singular_function('minbase')
I = S.ideal([z0**2, z0*z1, z1**2, z0*w0, 
       (1/2)*(z0*w1 + z1*w0), z1*w1, w0**2, w1*w0, w1**2])
M = I.syzygy_module()
F = matrix(S, 9,1, lambda i,j: I.gens()[i]);
v=matrix(S,4,1,[z0,z1,w0,w1])
N = matrix(S, 16, 9, lambda i, j: 'n'+ str(i) + str(j));
C = matrix(S, 9, 4, lambda i, j: 'c'+ str(i) + str(j));
D = matrix(S, 9, 1, lambda i, j: 'd'+ str(i) + str(j));
V = [z0**2, z0*z1, z1**2,w0**2, w1*w0, 
                  w1**2, z0*w0, z1*w1, z0*w1, z1*w0]
Z2 = M*C*v+N*F
a = [c32 + c43,c42 + c53, c72 + c83,  c62 + c73]
for i in range(Z2.nrows()):
	for j in range(len(V)):
		a.append(Z2[i][0].coefficient(V[j]))
J = S.ideal(a)
R=S.quotient_ring(J)
R.inject_variables()
N = matrix(R, 16, 9, lambda i, j: 'n'+ str(i) + str(j));
C = matrix(R, 9, 4, lambda i, j: 'c'+ str(i) + str(j));
Clift = matrix(S, C.nrows(), C.ncols(), lambda i, j: C[i][j].lift())
Nlift = matrix(S, N.nrows(), N.ncols(), lambda i, j: N[i][j].lift())

Z1 = Nlift*Clift*v + M*D
b=[]
for i in range(Z1.nrows()):
	for j in range(v.nrows()):
		b.append(Z1[i][0].coefficient(v[j][0]))
JJ = S.ideal(b)
\end{verbatim}

\bibliographystyle{amsalpha}
\bibliography{Lit}

\end{document}